\documentclass{amsart}
\usepackage{pifont}
\usepackage{amsfonts}

\usepackage{amscd,amssymb,amsmath,graphicx,verbatim}
\usepackage[dvips]{hyperref}
\usepackage[TS1,OT1,T1]{fontenc}

\newtheorem{theorem}{Theorem}[section]
\newtheorem{lemma}[theorem]{Lemma}
\newtheorem{corollary}[theorem]{Corollary}
\newtheorem{proposition}[theorem]{Proposition}

\theoremstyle{definition}
\newtheorem{definition}[theorem]{Definition}
\newtheorem{example}[theorem]{Example}

\newtheorem{question}[theorem]{Question}

\theoremstyle{remark}
\newtheorem{remark}[theorem]{Remark}

\numberwithin{equation}{section}

\makeatletter
\def\Eqlfill@{\arrowfill@\Relbar\Relbar\Relbar}
\newcommand{\extendEql}[1][]{\ext@arrow 0359\Eqlfill@{#1}}
\makeatother

\begin{document}

\title[Schauder Bases and Operator Theory II: (SI) Schauder Operators]{Schauder Bases and Operator Theory II: strongly irreducible Schauder Operators}

\author{Yang Cao}
\address{Yang Cao, Department of Mathematics , Jilin university, 130012, Changchun, P.R.China}\email{caoyang@jlu.edu.cn}

\author{Youqing Ji}
\address{Youqing Ji, Department of Mathematics , Jilin university, 130012, Changchun, P.R.China}\email{jiyq@jlu.edu.cn}

\author{Geng Tian}
\address{Geng Tian, Department of Mathematics , Jilin university, 130012, Changchun, P.R.China}\email{tiangeng09@mails.jlu.edu.cn}

\date{Feb. 14, 2012}
\subjclass[2000]{Primary 47A55, 47A53, 47A16; Secondary 54H20,
37B99.} \keywords{.}
\thanks{}
\begin{abstract}
In this paper, we will show that for an operator $T$ which is injective and has dense range, there exists an invertible operator $X$ (in fact we can find $U+K$, where $U$ is an unitary operator and $K$ is a compact operator with norm less than a given positive real number) such that $XT$ is strongly irreducible.
As its application, strongly irreducible operators always exist in the orbit of Schauder matrices.
\end{abstract}
\maketitle

\section{Introduction and preliminaries}

From the viewpoint of matrices, Schauder bases and operator theory have natural relations. For example,
the column vector sequence of the matrix of an invertible operator
comprise a Riesz basis under any orthonormal basis (ONB). In our series paper, we focus on the
operator which has a matrix representation consisting of a Schauder basis under appropriate ONB (We shall call them Schauder operators and Schauder matrices respectively).

In the matrix theory of finite dimensional space, the famous Jordan canonical form theorem sufficiently reveals the internal structure of matrices. Works of Jiang C. L., Ji Y. Q. etc. show that the strongly irreducible operators can be seen as the generalized Jordan block
in the case of separable infinite dimensional Hilbert space.
It will be shown that multiplying the matrix of any invertible operator on the left of a Schauder matrix gets an equivalent Schauder
matrix (Theorem \ref{BPS}). Since Jordan blocks are the objects with relatively simple properties, a natural question can be raised as follows:

\begin{question}\label{question}
Given any Schauder matrix $M$, does there exist a matrix $X$ of some invertible operator such that $XM$ is strongly irreducible?
\end{question}

In this paper, we shall give an affirmative answer to this question.
Following theorem is our main result in operator theory which will be used to solve the question.

\begin{theorem}\label{Theorem: XT can be strongly irreducible}
Let $T\in\mathcal{L}(\mathcal{H})$ satisfy $KerT=(0)$ and $\overline{RanT}=\mathcal{H}$. Then for any $\epsilon>0$, there exist an unitary operator $U$ and a compact operator $K$ with $||K||<\epsilon$ such that $(U+K)T$ is strongly irreducible. Moreover, since $\epsilon$ can be chosen small enough, $U+K$ could be invertible.
\end{theorem}

\begin{remark} This theorem does not only motivated by basis theory observations.  There exist some other observations. We summarize them as follows:

1. Physical backgrounds. In the Mathematics and Physics Interdisciplinary Seminar at Jilin University, Professor Hai-Jun Wang of Physics Department told us that multiplying an operator $X$ on the left of a Hamiltonian operator $T$ can be viewed as an evolution (or an evolving step) of a physics system which is encoded by the Hamiltonian operator $T$. Moreover, the strong  irreducibility of an operator $T$ ensures that there is no subsystem developing independently to the rest of the system. He also suggested us to consider the Schauder basis as a sequence of "superposition states". From this viewpoint, the special case of theorem \ref{Theorem: XT can be strongly irreducible} in which $T$ is a self-adjoint operator may be more interesting. It could be useful in characterizing evolution of observables.

2. Compare to the classical results in the theory of strongly irreducible operators.
D.A.Herrero very concerned about compact perturbations of strongly irreducible operators and asked that for any operator $T$ with connected spectrum and any $\epsilon>0$, does there exist a compact operator $K$ with $||K||<\epsilon$ such that $T+K$ is strongly irreducible? This question is the essential strengthen of the question asked by G.H.Gong.
Among many years, C.L.Jiang, Y.Q.Ji, Z.Y.Wang, S.H.Sun and S.Power obtain the affirmative answer step by step.
\begin{theorem}\cite{JSW,JJW1,JJW2,JJW3,JPW,Ji,JJ}\label{JJWPS}
Let $T\in\mathcal{L}(\mathcal{H})$, $\sigma(T)$ connected. Then for any $\epsilon>0$, there exists a compact operator $K\in\mathcal{L}(\mathcal{H})$ with $||K||<\epsilon$ such that $T+K$ is strongly irreducible.
\end{theorem}
Theorem \ref{Theorem: XT can be strongly irreducible} is a parallel consideration along the line of above theorem. Theorem \ref{Theorem: XT can be strongly irreducible} focus on the action of multiplying an operator on the left while theorem \ref{JJWPS} focus on the action of adding
an operator.

3. Basis theory observations. Operator which is injective and has dense range has a natural basis theory understanding,
that is the main topic of our next paper. We introduce the main result here:
\begin{theorem}\label{Theorem: main theorem 1 sec3}
Let $T \in\mathcal{L}(\mathcal{H})$, then $T$ is a Schauder operator if and only if $T$ is  injective and has dense range.
\end{theorem}
With this theorem in mind, question \ref{question} naturally stimulate us to consider it towards theorem \ref{Theorem: XT can be strongly irreducible}.
\end{remark}

We organize this paper as follows. The second section contains the proof of our main theorem \ref{Theorem: XT can be strongly irreducible}.
In the last section  we apply the main theorem into basis theory.
To generalize the class of bases which can be studied by
operator theory, we also introduce the \textsl{blowing up matrix} in that section.

\subsection{Notations}
Let $\mathcal{H}_1,\mathcal{H}_2,\mathcal{H}$ be complex separable Hilbert spaces. Denote by $\mathcal{L}(\mathcal{H}_1,\mathcal{H}_2)$ the set of all
bounded linear operators mapping $\mathcal{H}_1$ into $\mathcal{H}_2$. Denote by $\mathcal{K}(\mathcal{H}_1,\mathcal{H}_2)$ the subset
of $\mathcal{L}(\mathcal{H}_1,\mathcal{H}_2)$ of all compact operators. We simply write $\mathcal{L}(\mathcal{H})$ and $\mathcal{K}(\mathcal{H})$ instead
of $\mathcal{L}(\mathcal{H},\mathcal{H})$ and $\mathcal{K}(\mathcal{H},\mathcal{H})$ respectively. For $T\in\mathcal{L}(\mathcal{H}_1,\mathcal{H}_2)$, denote the kernel of
$T$ and the range of $T$ by Ker$T$ and Ran$T$ respectively.
Let $T\in\mathcal{L}(\mathcal{H})$, denote by $\sigma(T)$, $\sigma_p(T)$, $\sigma_e(T)$,
$\sigma_{lre}(T)$ and $\sigma_W(T)$ the spectrum, the point spectrum, the essential
spectrum, the Wolf spectrum and the Weyl spectrum of $T$
respectively. Denote by $\sigma_0(T)$ the set of isolated
points of $\sigma(T)\backslash\sigma_e(T)$. For $\lambda\in \rho_{s-F}(T)$ $(\extendEql{\mbox{def}}\mathbb{C}\backslash
\sigma_{lre}(T))$, ${\rm ind}(\lambda-T)={\rm dimKer}(\lambda-T)-{\rm
dimKer}(\lambda-T)^*$. Denote
$\rho_{s-F}^{(n)}(T)=\{\lambda\in\rho_{s-F}(T);~\text{ind}(\lambda-T)=n\}$,
where $-\infty \leq n\leq \infty$.  $T$ is said to be quasi-triangular if there is a sequence
$\{P_n\}_{n=1}^\infty$ of finite rank projections increasing to the unit operator $I$ with respect to
the strong operator topology such that $\lim\limits_{n\rightarrow\infty}||(I-P_n)TP_n||=0$. It is well-known
that $T$ is quasi-triangular if and only if ind$(T-\lambda)\geq0$ for all $\lambda\in\rho_{s-F}(T)$. $T$ is said
to be strongly irreducible if there are no nontrivial idempotents commuting with $T$. A Cowen-Douglas operator
is an operator $T$ satisfying the following conditions:

(1) There is a nonempty connected open subset $\Omega$ of
$\rho_{s-F}^{(n)}(T)$ for a natural
number n;

(2) $T-\lambda$ is surjective for each
$\lambda\in\Omega$;

(3) $\bigvee _{\lambda\in \Omega}{\rm ker}(T-\lambda)=\mathcal{H}$.

If the conditions above are satisfied, we shall write $T\in\mathcal{B}_n(\Omega
)$.

\section{Main Results}

\subsection{} First, let us
introduce some known results. Denote by $r(T)$ the spectral radius of $T$. It is well-known that $$r(T)=\lim\limits_{n\rightarrow\infty}||T^n||^{\frac{1}{n}}.$$
A bilateral weighted shift $T\in\mathcal{L}(\mathcal{H})$ is an operator that maps each vector in some orthonormal basis $\{e_n\}_{n\in\mathbb{Z}}$ into a scalar multiple of the next vector, $Te_n=\omega_ne_{n+1}$,
for all $n\in\mathbb{Z}$.

\begin{lemma}\cite{ALS}\label{shift}
1) If $T\in\mathcal{L}(\mathcal{H})$ is an invertible bilateral weighted shift, then the spectrum of $T$ is the annulus $\{\lambda\in\mathbb{C};[r(T^{-1})]^{-1}\leq |\lambda|\leq r(T)\}$.

2) If $T\in\mathcal{L}(\mathcal{H})$ is a bilateral weighted shift that is not invertible, then the spectrum of $T$ is the disc $\{\lambda\in\mathbb{C};|\lambda|\leq r(T)\}$.
\end{lemma}

\begin{lemma}\cite{JLY}\label{JLY}
Let $T$ be a bilateral weighted shift operator with weight sequence $\{\omega_n\}_{n\in\mathbb{Z}}$. Then $T$ is a Cowen-Douglas operator if and only if there exists a $\lambda_0\in\rho_F(T)$ such that ind$(T-\lambda_0)=1$.
\end{lemma}
\begin{remark}\label{remark}
In the above lemma, if there exists a $\lambda_0\in\rho_F(T)$ such that ind$(T-\lambda_0)=1$, then $T$ must belong to $\mathcal{B}_1(\Omega
)$ for some open connected subset of $\mathbb{C}$. One can see \cite{JLY} for details.
\end{remark}

\begin{lemma}\cite{JJW1}\label{B1}
Any operator $T\in\mathcal{B}_1(\Omega
)$ is strongly irreducible.
\end{lemma}

\begin{lemma}\label{bilateral shift}
Let $T\in\mathcal{L}(\mathcal{H})$ be a bilateral weighted shift with weight sequence $\{\omega_n\}_{n\in\mathbb{Z}}$, $\omega_n>0$ for $n\in\mathbb{Z}$ and max$\{\omega_n;n\geq0\}<$ min$\{\omega_n;n<0\}$. Then $T\in\mathcal{B}_1(\Omega
)$ for some open connected subset $\Omega$ of $\mathbb{C}$ and hence strongly irreducible.
\end{lemma}
\begin{proof}
Let $\mathcal{H}_1=\bigvee_{n\leq0}\{e_n\}$, $\mathcal{H}_2=\bigvee_{n\geq1}\{e_n\}$, then $$T=\begin{bmatrix}
\begin{bmatrix}
\ddots\\
\ddots&0\\
&\omega_{-1}&0\end{bmatrix}&0\\
\begin{bmatrix}\cdots&0&\omega_0\\
&0&0\\
{\mathinner{\mkern2mu\raise1pt\hbox{.}\mkern2mu
\raise4pt\hbox{.}\mkern2mu\raise7pt\hbox{.}\mkern1mu}}&&\vdots\\
\end{bmatrix}&\begin{bmatrix}0\\
\omega_1&0\\
&\ddots&\ddots\end{bmatrix}
\end{bmatrix}\begin{matrix}\vdots\\
e_{-1}\\
e_0\\
e_1\\
e_2\\
\vdots\end{matrix}\extendEql{\mbox{def}}\begin{bmatrix}A\\
C&B\end{bmatrix}\begin{matrix} \mathcal{H}_1\\
        \mathcal{H}_2\end{matrix}.$$

Choose any $\lambda$ such that max$\{\omega_n;n\geq0\}<\lambda<$ min$\{\omega_n;n<0\}$. We will show that $T-\lambda$ is a Fredholm operator with index 1.

On the one hand, ${\rm inf}_{x\neq0}\frac{||A^*x||}{||x||}\geq$ min$\{\omega_n;n<0\}>\lambda$, $||(A^*-\lambda)(x)||\geq||A^*x||-\lambda||x||\geq({\rm min}\{\omega_n;n<0\}-\lambda)||x||$ for any $x\in\mathcal{H}_1$, hence $(A-\lambda)^*$ is bounded from below and $A-\lambda$ is right invertible. Let $X_{11}\in\mathcal{L}(\mathcal{H}_1)$ satisfy $(A-\lambda)X_{11}=I$.

On the other hand, $r(B)=\lim\limits_{k\rightarrow\infty}({\rm sup}_{n\geq0}\{\omega_{n+1}\omega_{n+2}\cdots\omega_{n+k}\})^{\frac{1}{k}}<\lambda$, hence $B-\lambda$ is invertible.

Let $$X=\begin{bmatrix}X_{11}&0\\
-(B-\lambda)^{-1}CX_{11}&(B-\lambda)^{-1}\end{bmatrix}\begin{matrix} \mathcal{H}_1\\
        \mathcal{H}_2\end{matrix},$$
then $(T-\lambda)X=I$. It is easy to compute that Ker$(T-\lambda)=\{\alpha x;\alpha\in\mathbb{C}\}$, where \begin{eqnarray*}x=\sum_{n\in\mathbb{Z}}x_ne_n, ~x_0=1,~ x_n=\dfrac{\omega_0\omega_1\cdots\omega_{n-1}}{\lambda^n},~
x_{-n}=\dfrac{\lambda^{n}}{\omega_{-1}\omega_{-2}\cdots\omega_{-n}},~ \forall~ n>0.\end{eqnarray*}  Hence  dimKer$(T-\lambda)=1$ and $T-\lambda$ is a Fredholm operator with index 1. From lemma \ref{JLY}, remark \ref{remark}, lemma \ref{B1}, we know that $T\in\mathcal{B}_1(\Omega
)$ for some open connected subset $\Omega$ of $\mathbb{C}$ and strongly irreducible.
\end{proof}

\begin{lemma}\cite{HER1}\label{HER1}
Let $A\in\mathcal{L}(\mathcal{H}_1),B\in\mathcal{L}(\mathcal{H}_2)$,  denote by $\mathcal{T}_{A,B}$ the Rosenblum operator, then

(1) the following three statements equivalent:

i) Ran$\mathcal{T}_{A,B}=\mathcal{L}(\mathcal{H}_2,\mathcal{H}_1)$;

ii) $\sigma_r(A)\cap\sigma_l(B)=\emptyset$, where $\sigma_r(A)$ and $\sigma_l(B)$ are right spectrum and left spectrum respectively;

iii) $\mathcal{K}(\mathcal{H}_2,\mathcal{H}_1)\subseteq$ Ran$\mathcal{T}_{A,B}$.

(2) If $\sigma_l(A)\cap\sigma_r(B)=\emptyset$, then Ker$\mathcal{T}_{A,B}=\{0\}$.
\end{lemma}

\begin{lemma}\cite{JW}\label{JW}
Let $0<l_j\leq\infty, j=1,2,$ $\mathcal{H}=\oplus_{j=1}^2(\oplus_{0\leq i<l_j}\mathcal{H}_i^j)$, then $$T=\begin{bmatrix}
A&Q\\
&B\end{bmatrix}\begin{matrix}\oplus_{0\leq i<l_1}\mathcal{H}_i^1\\
\oplus_{0\leq i<l_2}\mathcal{H}_i^2\end{matrix}\in\mathcal{L}(\mathcal{H})$$ is a strongly irreducible operator when the following conditions satisfy:

i) $A=\oplus_{0\leq i<l_1}A_i,B=\oplus_{0\leq i<l_2}B_i, Q=(Q_{ij})_{ij}$, where $A_i\in\mathcal{L}(\mathcal{H}_i^1),B_i\in\mathcal{L}(\mathcal{H}_i^2)$, $Q_{ij}\in\mathcal{L}(\mathcal{H}_j^2,\mathcal{H}_i^1)$;

ii) $A_i,B_j$ are all strongly irreducible operators;

iii) Ker$\mathcal{T}_{B_j,A_i}=\{0\}$ for all $i,j$, Ker$\mathcal{T}_{A_i,A_j}=\{0\}$, Ker$\mathcal{T}_{B_i,B_j}=\{0\}$ for all $i,j,i\neq j$;

iv) If $\sigma_r(A_i)\cap\sigma_l(B_j)\neq\emptyset$, then $Q_{ij}\overline{\in}Ran\mathcal{T}_{A_i,B_j}$;

v) For any $0\leq m,n<l_1$, there exist $k_0<\infty$, $\{i_k\}_{k=1}^{k_0+1}$, $\{j_k\}_{k=1}^{k_0}$ such that $\sigma_r(A_{i_k})\cap\sigma_l(B_{j_k})\neq\emptyset$, $\sigma_r(A_{i_{k+1}})\cap\sigma_l(B_{j_k})\neq\emptyset$, $i_1=m$, $i_{k_0+1}=n$;

vi) For any $0\leq m,n<l_2$, there exist $k_1<\infty$, $\{i_k\}_{k=1}^{k_1}$, $\{j_k\}_{k=1}^{k_1+1}$ such that $\sigma_r(A_{i_k})\cap\sigma_l(B_{j_k})\neq\emptyset$, $\sigma_r(A_{i_{k}})\cap\sigma_l(B_{j_{k+1}})\neq\emptyset$, $j_1=m$, $j_{k_1+1}=n$.
\end{lemma}

\begin{lemma}\cite{JiBoshilunwen,Ji}\label{Ji}
Let $T\in\mathcal{L}(\mathcal{H})$ be quasi-triangular, $\sigma(T),\sigma_W(T)$ are all connected and $\lambda_0\in\partial\sigma_W(T)$.  Then
for any given $\epsilon>0$, there exists a compact operator $K\in\mathcal{L}(\mathcal{H})$, $||K||<\epsilon$ such that

i) $T+K$ is strongly irreducible;

ii) If $\lambda_0\overline{\in}\sigma_p(B)$, then Ker$\mathcal{T}_{B,T+K}=\{0\}$.
\end{lemma}

\subsection{Proof of Main Theorem}
\begin{proof}
From the polar decomposition theorem, we have $T=V|T|.$  Since $KerT=(0)$ and $\overline{RanT}=\mathcal{H}$, $V$ is  an  unitarily operator. If the theorem is correct for positive operator $|T|$, i.e. there exist unitarily operator $U$ and compact operator $K$, $||K||<\epsilon$ such that $(U+K)|T|$ is strongly irreducible, then $(UV^*+KV^*)T=(U+K)|T|$ is strongly irreducible and the norm of compact operator $KV^*$ is less than $\epsilon$. Hence it suffices to prove the theorem when T is a positive operator. We will break it into four cases.

Case 1. Let $T$ be an invertible operator. Since $\sigma_e(T)\neq\emptyset$, we denote $\lambda_{min}=min\{\lambda;\lambda\in\sigma_e(T)\},\lambda_{min}=min\{\lambda;\lambda\in\sigma_e(T)\}$. From the Weyl-Von Neumann theorem, for $\epsilon>0$, there exists a compact operator $K_1$, $||K_1||<\frac{\epsilon}{2||T^{-1}||}$ such that,

1) $T+K_1$ is a diagonal operator with entries $\{\lambda_1,\lambda_2,\ldots\}$ under an ONB $\{e_n\}_{n=1}^\infty$;

2) there exist two subsequences $\{\alpha_n\}_{n=1}^\infty$ and $\{\beta_n\}_{n=1}^\infty$ of $\{\lambda_n\}_{n=1}^\infty$ which satisfy \begin{eqnarray*}&&\lim\limits_{n\rightarrow\infty}\alpha_n=\lambda_{min},~\alpha_n<\alpha_{n+1},\\ &&\lim\limits_{n\rightarrow\infty}\beta_n=\lambda_{max},~\beta_n>\beta_{n+1};\end{eqnarray*}

3) $\{\lambda_1,\lambda_2,\ldots\}\backslash\{\{\alpha_1,\alpha_2,\ldots\}\bigcup\{\beta_1,\beta_2,\ldots\}\}\subseteq[\lambda_{min},\lambda_{max}]$.

Denote the elements of set $\{\lambda_1,\lambda_2,\ldots\}\backslash\{\{\alpha_1,\alpha_2,\ldots\}\bigcup\{\beta_1,\beta_2,\ldots\}\}$ by $\{\gamma_n\}_{n=1}^\infty$. Then we have

$$T+K_1=
  \begin{bmatrix}
      \begin{bmatrix}
      \ddots &  &  &  &  &  \\
     & \beta_2&  \\
     & & \beta_1& \\
     &  &  & \alpha_1 &  \\
     &  &  &  & \alpha_2 & \\
     &  &  &  &  & \ddots \\
      \end{bmatrix}
     & \\
     & \begin{bmatrix}
           \gamma_1 \\
            &\gamma_2 \\
            &&\ddots \\
         \end{bmatrix}\extendEql{\mbox{def}}D
      \\
  \end{bmatrix}
\begin{matrix}
  \vdots \\
  f_4 \\
  f_2 \\
  f_1 \\
  f_5 \\
  \vdots \\
  f_3 \\
  f_7 \\
  \vdots\\
\end{matrix}
,
$$ where $\{f_n\}_{n=1}^\infty$ is a rearrangement of ONB $\{e_n\}_{n=1}^\infty$.

Let $$U= \begin{bmatrix}
           \ddots \\
            \ddots& 0 &  \\
            & 1& 0 &  \\
            & & 1& 0 &  \\
            & & & 1& 0 &  \\
            & & & & \ddots& \ddots &  \\
        \end{bmatrix}
\begin{matrix}
  \vdots \\
  f_4 \\
  f_2 \\
  f_1 \\
  f_5 \\
  \vdots \\
\end{matrix}
~{\rm and}~U_1=\begin{bmatrix}
          U &  \\
           & I \\
        \end{bmatrix}\begin{matrix}
        \mathcal{H}_1\\
        \mathcal{H}_2
        \end{matrix},
$$ where $\mathcal{H}_1=\bigvee\{f_{2n+2},f_{4n+1},n=0,1,2,\ldots\}$, $\mathcal{H}_2=\bigvee\{f_{4n+3},n=0,1,2,\ldots\}$,
then we have $$U_1(T+K_1)=
  \begin{bmatrix}
      \begin{bmatrix}
    \ddots   \\
      \ddots& 0 \\
     & \beta_2 &0&\\
     &  &  \beta_1&0\\
     &  &  &\alpha_1&0 \\
     &&&&\ddots&\ddots\\
      \end{bmatrix}\extendEql{\mbox{def}}B
     &  \\
     & D
      \\
  \end{bmatrix}
\begin{matrix}
  \vdots \\
  f_4 \\
  f_2 \\
  f_1 \\
  f_5 \\
  \vdots \\
\mathcal{H}_2
\end{matrix}.$$

It is easy to compute that \begin{eqnarray*}&&r(B)=\lim\limits_{n\rightarrow\infty}(\beta_1\beta_2\cdots\beta_n)^{\frac{1}{n}}=\lim\limits_{n\rightarrow\infty}\beta_n=\lambda_{max},\\
&&r(B^{-1})^{-1}=\lim\limits_{n\rightarrow\infty}((\alpha_1^{-1}\alpha_2^{-1}\cdots\alpha_n^{-1})^{\frac{1}{n}})^{-1}=\lim\limits_{n\rightarrow\infty}\alpha_n=\lambda_{min}.\end{eqnarray*}

From lemma \ref{shift}, the spectrum of $B$ is the annulus $\{\lambda\in\mathbb{C};\lambda_{min}\leq|\lambda|\leq\lambda_{max}\}$. Moreover, the diagonal operator $D$ is contained in this annulus. Hence the spectrum $\sigma(U_1(T+K_1))$ is connected. It follows from the theorem \ref{JJWPS} that there exists a compact operator $K_2$, $||K_2||<\frac{\epsilon}{2||T^{-1}||}$ such that $U_1(T+K_1)+K_2$ is strongly irreducible. Notice that $\{U_1+(U_1K_1T^{-1}+K_2T^{-1})\}T=U_1(T+K_1)+K_2$ and $U_1K_1T^{-1}+K_2T^{-1}$ is a compact operator with $||U_1K_1T^{-1}+K_2T^{-1}||<\epsilon$, we complete the proof of this case.

Case 2. Let $T$ be an operator with $\sigma_e(T)=\{0\}$. Since $KerT=(0)$ and $\overline{RanT}=\mathcal{H}$, $T$ must have the matrix representation as follows, $$T=\begin{bmatrix}
                 \lambda_0 & \\
                  &\lambda_1 \\
                  && \lambda_{-1}\\
                  &&&\lambda_2\\
                &&&&\lambda_{-2}\\
                &&&&&\ddots \\
               \end{bmatrix}
             \begin{matrix}
                      e_0 \\
                      e_1 \\
                      e_2 \\
                     e_3 \\
                      e_4 \\
                      \vdots
                    \end{matrix},
$$ where $\lambda_n\neq0$, $\lim\limits_{n\rightarrow +\infty}\lambda_n=\lim\limits_{n\rightarrow-\infty}\lambda_n=0.$  As case 1, we can find an unitarily operator $U$ such that $$UT=
               \begin{bmatrix}
                 \ddots\\
                 \ddots&0\\
                &\lambda_{-1}&0\\
                &&\lambda_0 &0\\
                &&&\lambda_1&0\\
                &&&&\ddots&\ddots\\
               \end{bmatrix}
             \begin{matrix}
                      \vdots \\
                      e_2 \\
                      e_0 \\
                      e_1 \\
                      e_3 \\
                      \vdots
                   \end{matrix}.$$

We will show that $UT$ is strongly irreducible by directly computation. Let $$P=\begin{bmatrix}
                 \ddots&\vdots&\vdots&\vdots&\vdots&

{\mathinner{\mkern2mu\raise1pt\hbox{.}\mkern2mu
\raise4pt\hbox{.}\mkern2mu\raise7pt\hbox{.}\mkern1mu}}
 \\
                 \cdots&a_{-1,-1}&a_{-1,0}&a_{-1,1}&a_{-1,2}&\cdots\\
                \cdots&a_{0,-1}&a_{0,0}&a_{0,1}&a_{0,2}&\cdots\\
                \cdots&a_{1,-1}&a_{1,0} &a_{1,1}&a_{1,2}&\cdots\\
                \cdots&a_{2,-1}&a_{2,0}&a_{2,1}&a_{2,2}&\cdots\\

{\mathinner{\mkern2mu\raise1pt\hbox{.}\mkern2mu
\raise4pt\hbox{.}\mkern2mu\raise7pt\hbox{.}\mkern1mu}}

&\vdots&\vdots&\vdots&\vdots&\ddots\\
              \end{bmatrix}
             \begin{matrix}
                      \vdots \\
                      e_2 \\
                      e_0 \\
                      e_1 \\
                      e_3 \\
                      \vdots
                   \end{matrix}$$ be an operator commuting with $UT$. Then we have \begin{eqnarray}&&\lambda_ia_{i,j}=\lambda_ja_{i+1,j+1},~\forall i,j\in\mathbb{Z}.\end{eqnarray}

First, $a_{i,j}=0$ when $i\neq j$. If $k>0$, then \begin{eqnarray} &&a_{i,i+k}=\dfrac{\lambda_{k-1}\lambda_{k-2}\cdots\lambda_0}{\lambda_{i+k-1}\lambda_{i+k-2}\cdots\lambda_i}a_{0,k},~ \forall~ i\geq k,\\
&&a_{i,i+k}=\dfrac{\lambda_{i+k}\lambda_{i+k+1}\cdots\lambda_{k-1}}{\lambda_i\lambda_{i+1}\cdots\lambda_{-1}}a_{0,k},~ \forall~ i\leq -1.\end{eqnarray}

If $k<0$, then
\begin{eqnarray}&&a_{i,i+k}=\dfrac{\lambda_{k+1}\lambda_{k+2}\cdots\lambda_0}{\lambda_{i+k-1}\lambda_{i+k}\cdots\lambda_{i-2}}a_{0,k},~ \forall~ i\leq k+2,\\
&&a_{i,i+k}=\dfrac{\lambda_{i-1}\lambda_{i-2}\cdots\lambda_0}{\lambda_{i+k-1}\lambda_{i+k-2}\cdots\lambda_k}a_{0,k},~ \forall~ i\geq 1.\end{eqnarray}

From $(2.2),(2.4)$, we know that $a_{0,k}=0$ for all $k\in\mathbb{Z}$. Hence $a_{i,j}=0$ when $i\neq j$.

Second, $a_{i,i}=a_{j,j}$ for all $i,j\in\mathbb{Z}$. It is easily seen from $(2.1)$. In summary, $P=I$ or $P=0$.

Case 3. Let $T$ be an operator which satisfy that $0$ is an isolate point of $\sigma_e(T)$ and $\sigma_e(T)\backslash\{0\}\neq\emptyset$.

From the spectral decomposition theorem of self-adjoint operators, we have $$T=\begin{bmatrix}
                                                                       T_1 &  \\
                                                                        & T_2 \\
                                                                     \end{bmatrix}
                                                                   \begin{matrix}
                                                                            \mathcal{H}_1 \\
                                                                            \mathcal{H}_2
                                                                          \end{matrix},
$$ where $\sigma_e(T_1)=\{0\}$ and $0\overline{\in}\sigma(T_2)$, $\mathcal{H}_1 \bot\mathcal{H}_2$ and $\mathcal{H}_1\oplus\mathcal{H}_2=\mathcal{H}$.

Notice that dim$\mathcal{H}_2=\infty$.

As case 2, since $KerT=(0)$ and $\overline{RanT}=\mathcal{H}$, $T_1$ must be a compact operator and $$T_1=
        \begin{bmatrix}
            \lambda_0\\
            & \lambda_1 \\
           &  & \lambda_2  \\
           &&&\ddots
        \end{bmatrix}
      \begin{matrix}
               e_0 \\
               e_1 \\
              e_2 \\
               \vdots
             \end{matrix}
,$$
where $\{e_n\}_{n=0}^\infty$ is an ONB of $\mathcal{H}_1$, $\lambda_n\neq0$, $\lim\limits_{n\rightarrow\infty}\lambda_n=0.$

Since $T_2$ is an invertible self-adjoint operator,  there exist a compact operator $K$ with $||K||<\frac{\epsilon}{||T_2^{-1}||}$, such that $T_2+K$ is also invertible, and $$T_2+K=
        \begin{bmatrix}
             \eta_1 \\
             & \eta_2  \\
             &  & \eta_3 \\
           &&&\ddots
        \end{bmatrix}
      \begin{matrix}
               f_1 \\
              f_2 \\
               f_3 \\
               \vdots
            \end{matrix},$$  where $\{f_n\}_{n=1}^\infty$ is an ONB of $\mathcal{H}_2$.

Let $$ K_1=        \begin{bmatrix}
                              0 &  \\
                               & K \\
                            \end{bmatrix}
                            \begin{matrix}
                            \mathcal{H}_1 \\
                             \mathcal{H}_2
                             \end{matrix}.$$

Choose $U$ as case 1 such that $$U(T+K_1)=\begin{bmatrix}
           \ddots \\
             \ddots& 0 \\
             &\eta_1 & 0 \\
             && \lambda_0& 0 \\
            &&&\lambda_1& 0 & \\
          &&&& \ddots&\ddots\\
\end{bmatrix}
  \begin{matrix}
  \vdots \\
  f_1 \\
  e_0 \\
  e_1 \\
  e_2\\
  \vdots \\
\end{matrix}.$$

From lemma \ref{bilateral shift}, we know that $U(T+K_1)$ is strongly irreducible.

Let $K_2=UK_1T^{-1},
$ then $||K_2||<\epsilon$ and $(U+K_2)T=U(T+K_1)$ is strongly irreducible. So we complete the proof of this case.

Case 4. This case is a little harder. Let $T$ be an operator which satisfy that $0$ is not an isolate point of $\sigma_e(T)$, i.e. there exists a sequence $\lambda_n$ such that $\{\lambda_n\}_{n=1}^\infty\subseteq\sigma_e(T)$, $\lambda_n>\lambda_{n+1}$ and $\lim\limits_{n\rightarrow\infty}\lambda_n=0.$

From the spectral decomposition theorem of self-adjoint operators, we have $$T=\begin{bmatrix}
                                     T_1  \\
                                     &T_2 \\
                                   &&T_3\\
                                     &&&\ddots \\
                                  \end{bmatrix}
                                 \begin{matrix}
                                          \mathcal{H}_1 \\
                                          \mathcal{H}_2 \\
                                          \mathcal{H}_3 \\
                                          \vdots
                                    \end{matrix},
$$ where $\{\mathcal{H}_j\}_{j\geq 1}$ is a pairwise orthogonal family of subspaces and $\bigoplus_{j\geq1}\mathcal{H}_j=\mathcal{H}$, $\sigma(T_j)\subseteq[\lambda_{j+1},\lambda_j]$, $\lambda_{j+1},\lambda_j\in\sigma_e(T_j)$ for all $j\geq 1$.

First step, we will find unitarily operators and compact operators $U_j$, $K_j$, $||K_j||<\frac{\epsilon}{2j}$ such that each of $(U_j+K_j)T_j$ is strongly irreducible.

1) If there exists a point $\lambda_{j+1}<\alpha_j<\lambda_j$ such that $\alpha_j-T_j$ is invertible, then we deal with it as follows. Let $[\alpha_j-\delta_j,\alpha_j+\delta_j]$ be the small interval contained in $(\lambda_{j+1},\lambda_j)$ such that $[\alpha_j-\delta_j,\alpha_j+\delta_j]\subseteq\rho(T_j)$. From the Weyl-Von Neumann theorem, there exists a compact operator $K_j$, $||K_j||<\frac{\epsilon}{2j||T_j^{-1}||}$ such that $$T_j+K_j=                                                   \begin{bmatrix}
                                                                                                                     \ddots \\
                                                                                                                     &&\beta_2^j \\
                                                                                                                    &&&\beta_0^j \\
                                                                                                                     &&&&\beta_1^j\\
                                                                                                                     &&&&&&\ddots\\
                                                                                                                   \end{bmatrix}
                                                                                                                \begin{matrix}
                                                                                                                  \vdots \\
                                                                                                                     e_2^j \\
                                                                                                                        e_0^j \\
                                                                                                                         e_1^j \\
                                                                                                                         \vdots
                                                                                                                       \end{matrix},
$$ where $\{\beta_{2n}^j\}\subseteq(\alpha_j+\delta_j,\lambda_j],\{\beta_{2n+1}^j\}\subseteq[\lambda_{j+1},\alpha_j-\delta_j)$ and $\bigvee_{i\geq0}\{e_i^j\}=\mathcal{H}_j$.

Let $U_j\in\mathcal{L}(\mathcal{H}_j)$ as case 1 such that $$U_j(T_j+K_j)=\begin{bmatrix}
                                              \ddots &  \\
                                             \ddots &0&\\
                                             &\beta_2^j&0& \\
                                             && \beta_0^j&0\\
                                             &&&\beta_1^j&0\\
                                             &&&&\ddots&\ddots\\
                                          \end{bmatrix}
                   \begin{matrix}
                      \vdots \\
                      e_2^j \\
                      e_0^j \\
                      e_1^j \\
                      e_3^j \\
                      \vdots
                     \end{matrix}\extendEql{\mbox{def}}B_j.
$$

Then from lemma \ref{bilateral shift}, $B_j$ is a bilateral shift operator which is strongly irreducible. Denote $\widetilde{K_j}=U_jK_jT_j^{-1}$, then $||\widetilde{K_j}||<\frac{\epsilon}{2j}$ and $(U_j+\widetilde{K_j})T_j=B_j$ is strongly irreducible.

2) If $[\lambda_{j+1},\lambda_j]\subseteq\sigma(T_j)$, then we deal with it as follows. Choose an interval $[\eta_{j}^2,\eta_j^1]\subseteq(\lambda_{j+1},\lambda_j)$ arbitrarily. Then $$T_j=  \begin{bmatrix}
                                                                                    T_{j1} &  \\
                                                                                     & T_{j2} \\
                                                                                     &&T_{j3} \\
                                                                                  \end{bmatrix}
                                                                                \begin{matrix}
                                                                                         \mathcal{H}_j^1 \\
                                                                                         \mathcal{H}_j^2 \\
                                                                                         \mathcal{H}_j^3
                                                                                       \end{matrix},
$$ where $\{\mathcal{H}_j^i\}_{i=1}^3$ is pairwise orthogonal family of subspaces of $\mathcal{H}_j$ and $\bigoplus_{i=1}^3\mathcal{H}_j^i=\mathcal{H}_j$, $\sigma(T_{j1})=[\eta_j^1,\lambda_j]$, $\sigma(T_{j2})=[\lambda_{j+1},\eta_{j}^2]$, $\sigma(T_{j3})=[\eta_{j}^2,\eta_j^1]$. From the Weyl-Von Neumann theorem, there exist compact operators $C_{j1}\in\mathcal{L}(\mathcal{H}_j^1)$, $C_{j2}\in\mathcal{L}(\mathcal{H}_j^2)$, $||C_{j1}||<\frac{\epsilon}{4j||T_j^{-1}||},||C_{j2}||<\frac{\epsilon}{4j||T_j^{-1}||}$ such that, $T_{j1}+C_{j1}$ is a diagonal operator with entries $\{\gamma_{n}^j\}_{n\leq0}\subseteq[\eta_j^1,\lambda_j]$, $T_{j2}+C_{j2}$ is a diagonal operator with entries $\{\gamma_{n+1}^j\}_{n\geq0}\subseteq[\lambda_{j+1},\eta_{j}^2]$, Card$\{n;\gamma_{n+1}^j=\eta_{j}^2\}=\infty$ and \begin{eqnarray}&&\dfrac{\gamma_0^j\gamma_1^j\cdots\gamma_{n-1}^j}{(\eta_{j}^2)^n}\geq\dfrac{1}{\sqrt{n}},~\forall n\geq1.\end{eqnarray}

As case 1, denote $V_j\in\mathcal{L}(\mathcal{H}_j^1\oplus\mathcal{H}_j^2)$ be the unitarily operator such that
$$V_j\begin{bmatrix}
         T_{j1}+C_{j1} &  \\
          & T_{j2}+C_{j2} \\
       \end{bmatrix}
     =\begin{bmatrix}
          \ddots &  \\
           \ddots& 0 \\
          & \gamma_{-1}^j& 0\\
          && \gamma_0^j& 0  \\
          &&& \gamma_1^j& 0  \\
          &&&& \ddots & \ddots \\
        \end{bmatrix}\begin{matrix}\vdots\\
        e_{-1}^j\\
        e_{0}^j\\
        e_1^j\\
        e_2^j\\
        \vdots
        \end{matrix}\extendEql{\mbox{def}}D_j,
$$ where $\bigvee_{i\in\mathbb{Z}}\{e_i^j\}=\mathcal{H}_j^1\oplus\mathcal{H}_j^2$.

From lemma \ref{bilateral shift}, $D_j$ is strongly irreducible and moreover $\eta_{j+1}\overline{\in}\sigma_p(D_j)$. If $(D_j-\eta_{j}^2)(x)=0$, $x=\Sigma_{i\in\mathbb{Z}}x_ie_i^j\in\mathcal{H}_j^1\oplus\mathcal{H}_j^2$, then for $n\geq0$ \begin{eqnarray}&&x_{n+1}=\dfrac{\gamma_0^j\gamma_1^j\cdots\gamma_{n}^j}{(\eta_{j}^2)^{n+1}}x_0,~
x_{-n}=\dfrac{(\eta_{j}^2)^{n}}{\gamma_{-1}^j\gamma_{-2}^j\cdots\gamma_{-n}^j}x_0.
\end{eqnarray}
From $(2.6)$ and $(2.7)$, $x_n=0,~\forall n\in\mathbb{Z}$.

Since $T_{j3}$ is a quasi-triangular operator, $\sigma(T_{j3})=\sigma_W(T_{j3})=[\eta_{j}^2,\eta_j^1]$ and $\eta_{j}^2\in\partial\sigma_W(T_{j3})$, from lemma \ref{Ji}, there exists a compact operator $C_{j3}$, $||C_{j3}||<\frac{\epsilon}{8j||T_j^{-1}||}$ such that $T_{j3}+C_{j3}$ is strongly irreducible and Ker$(\mathcal{T}_{D_j,T_{j3}+C_{j3}})=\{0\}$. Moreover, since $\sigma_r(T_{j3}+C_{j3})\cap\sigma_l(D_j)\neq\emptyset$, from lemma \ref{HER1}, there exists a compact operator $E_j:  \mathcal{H}_j^1\oplus\mathcal{H}_j^2\longrightarrow \mathcal{H}_j^3$, $||E_j||<\frac{\epsilon}{8j||T_j^{-1}||}$ such that $E_j\overline{\in}Ran\mathcal{T}_{T_{j3}+C_{j3},D_j}$.

Hence from lemma \ref{JW}, we obtain $$S_j\extendEql{\mbox{def}}\begin{bmatrix}
                                              D_j &  \\
                                              E_j& T_{j3}+C_{j3} \\
                                            \end{bmatrix}
                                          \begin{matrix}
                                 \mathcal{H}_j^1\oplus\mathcal{H}_j^2 \\
                                 \mathcal{H}_j^3
                               \end{matrix}
$$ is strongly irreducible.

 Let $$W_j=      \begin{bmatrix}
                            V_j &  \\
                             & I \\
                          \end{bmatrix}
                        \begin{matrix}
                                 \mathcal{H}_j^1\oplus\mathcal{H}_j^2 \\
                                 \mathcal{H}_j^3
                               \end{matrix},
$$ and $$C_j=W_j\begin{bmatrix}
  C_{j1}\\
  &C_{j2}\\
   &&0 \\
  \end{bmatrix}
T_j^{-1}+ \begin{bmatrix}
                     0 & 0 \\
                     E_j & C_{j3} \\
                   \end{bmatrix}
                   T_j^{-1},$$
then $C_j$ is a compact operator with $||C_j||<\frac{\epsilon}{2j}$ and $(W_j+C_j)T_j=S_j$ is strongly irreducible.

In summary, we obtain $(U_j+\widetilde{K_j})T_j=B_j$ or $(W_j+C_j)T_j=S_j$ for any chosen $T_j$. For the sake of brevity, we denote $U_j$ or $W_j$ by $\overline{U}_j$, denote $\widetilde{K_j}$ or $C_j$ by $\overline{K}_j$, denote $B_j$ or $S_j$ by $\overline{B}_j$.

Let $U=\bigoplus_{j\geq1}\overline{U}_j$, $K=\bigoplus_{j\geq1}\overline{K}_j$, then $U$ is an unitarily operator, $K$ is a compact operator with $||K||<\frac{\epsilon}{2}$ and $$(U+K)T= \begin{bmatrix}
                                            \overline{B}_1\\
                                             & \overline{B}_2 \\
                                            & & \overline{B}_3  \\
                                            && & \ddots \\
                                          \end{bmatrix}
                                        \begin{matrix}
                                          \mathcal{H}_1 \\
                                          \mathcal{H}_2 \\
                                          \mathcal{H}_3 \\
                                          \vdots
                                        \end{matrix}.
$$

Second step, we will use lemma \ref{JW} to complete the proof. Since $\sigma_r(\overline{B}_{2n+1})\cap\sigma_l(\overline{B}_{2n})\neq\emptyset$, $\sigma_r(\overline{B}_{2n+1})\cap\sigma_l(\overline{B}_{2n+2})\neq\emptyset$, there exist compacts operator $E_n$, $F_n$, $||E_n||<\frac{\epsilon}{4n||T_{2n}^{-1}||}$, $||F_n||<\frac{\epsilon}{4n||T_{2n}^{-1}||}$ such that $E_n\overline{\in}$ Ran$\mathcal{T}_{\overline{B}_{2n-1},\overline{B}_{2n}}$ and
$F_n\overline{\in}$
Ran$\mathcal{T}_{\overline{B}_{2n+1},\overline{B}_{2n}}$.

Let  \begin{eqnarray*}\overline{K}&=&K+\begin{bmatrix}
0&\begin{bmatrix}E_1T_2^{-1}\\
 F_1T_2^{-1}&E_2T_4^{-1}\\
 &F_2T_4^{-1}&E_3T_6^{-1}\\
 &&\ddots&\ddots\\
\end{bmatrix}\\
&0
\end{bmatrix}
\begin{matrix}\mathcal{H}_1 \\
 \mathcal{H}_3 \\
 \mathcal{H}_5 \\
 \vdots\\
 \bigoplus_{n\geq1}\mathcal{H}_{2n} \\
\end{matrix}\\
&=&\begin{bmatrix}
\begin{bmatrix}
 \overline{K}_1\\
 & \overline{K}_3\\
 && \overline{K}_5\\
 &&& \ddots\\
 \end{bmatrix}&\begin{bmatrix}E_1T_2^{-1}\\
 F_1T_2^{-1}&E_2T_4^{-1}\\
 &F_2T_4^{-1}&E_3T_6^{-1}\\
 &&\ddots&\ddots\\
\end{bmatrix}\\
&\begin{bmatrix}
 \overline{K}_2\\
 & \overline{K}_4\\
 && \overline{K}_6\\
 &&& \ddots\\
 \end{bmatrix}
\end{bmatrix}
\begin{matrix}\mathcal{H}_1 \\
 \mathcal{H}_3 \\
 \mathcal{H}_5 \\
 \vdots\\
 \mathcal{H}_2 \\
 \mathcal{H}_4 \\
 \mathcal{H}_6 \\
 \vdots\end{matrix}, \end{eqnarray*}
 then $\overline{K}$ is a compact operator with $||\overline{K}||<\epsilon$ and
 \begin{eqnarray*}&&(U+\overline{K})T\\
 &=&\begin{bmatrix}
\begin{bmatrix}\begin{smallmatrix}
 \overline{U}_1+\overline{K}_1\\
 & \overline{U}_3+ \overline{K}_3\\
 &&  \overline{U}_5+\overline{K}_5\\
 &&& \ddots\\
 \end{smallmatrix}
 \end{bmatrix}&\begin{bmatrix}\begin{smallmatrix}E_1T_2^{-1}\\
 F_1T_2^{-1}&E_2T_4^{-1}\\
 &F_2T_4^{-1}&E_3T_6^{-1}\\
 &&\ddots&\ddots\\
 \end{smallmatrix}
\end{bmatrix}\\
&\begin{bmatrix}\begin{smallmatrix}
  \overline{U}_2+\overline{K}_2\\
 &  \overline{U}_4+\overline{K}_4\\
 &&  \overline{U}_6+\overline{K}_6\\
 &&& \ddots\\
\end{smallmatrix}\end{bmatrix}
\end{bmatrix}
\begin{matrix}\begin{smallmatrix}\mathcal{H}_1 \\
 \mathcal{H}_3 \\
 \mathcal{H}_5 \\
 \vdots\\
 \mathcal{H}_2 \\
 \mathcal{H}_4 \\
 \mathcal{H}_6 \\
 \vdots\end{smallmatrix}\end{matrix}\\
 &\begin{array}{c}
\times
\end{array}&
\begin{bmatrix}
\begin{bmatrix}
 T_1\\
 & T_3\\
 &&  T_5\\
 &&& \ddots\\
 \end{bmatrix}&0\\
0&\begin{bmatrix}
T_2\\
 & T_4\\
 &&T_6\\
 &&& \ddots\\
\end{bmatrix}
\end{bmatrix}
\begin{matrix}\mathcal{H}_1 \\
 \mathcal{H}_3 \\
 \mathcal{H}_5 \\
 \vdots\\
 \mathcal{H}_2 \\
 \mathcal{H}_4 \\
 \mathcal{H}_6 \\
 \vdots\end{matrix}\\
 &=&\begin{bmatrix}
\begin{bmatrix}
 \overline{B}_1\\
 & \overline{B}_3\\
 && \overline{B}_5\\
 &&& \ddots\\
 \end{bmatrix}&\begin{bmatrix}E_1\\
 F_1&E_2\\
 &F_2&E_3\\
 &&\ddots&\ddots\\
\end{bmatrix}\\
&\begin{bmatrix}
 \overline{B}_2\\
 & \overline{B}_4\\
 && \overline{B}_6\\
 &&& \ddots\\
 \end{bmatrix}
\end{bmatrix}
\begin{matrix}\mathcal{H}_1 \\
 \mathcal{H}_3 \\
 \mathcal{H}_5 \\
 \vdots\\
 \mathcal{H}_2 \\
 \mathcal{H}_4 \\
 \mathcal{H}_6 \\
 \vdots\end{matrix}.
\end{eqnarray*}

In order to obtain the result, we only need to show the conditions $i)-vi)$ of lemma \ref{JW}.

From the construction above, conditions $i),ii),iv),v),vi)$ are satisfied. Since $\sigma(\overline{B}_{i})\cap\sigma(\overline{B}_{j})=\emptyset$, when $|i-j|\geq2$, from lemma \ref{HER1}, Ker$\mathcal{T}_{\overline{B}_{i},\overline{B}_{j}}=\{0\}$ for $i\neq j$, $|i-j|\geq2$.

For $j$ even, $|i-j|=1$, we should consider Ker$\mathcal{T}_{\overline{B}_{j},\overline{B}_{i}}=\{0\}$.

Since each $\overline{B}_{j}$ maybe $B_j$ or $S_j$, we should consider four cases for any Ker$\mathcal{T}_{\ast,\ast}$. But the proof are all the same, we just consider one case that Ker$\mathcal{T}_{\overline{B}_{j},\overline{B}_{j+1}}=\{0\}$ where
\begin{eqnarray*}&\overline{B}_{j}=S_j=\begin{bmatrix}       D_j &  \\
                                              E_j& T_{j3}+C_{j3} \\
                                            \end{bmatrix}
                                          \begin{matrix}
                                 \mathcal{H}_j^1\oplus\mathcal{H}_j^2 \\
                                 \mathcal{H}_j^3
                               \end{matrix},\\
&\overline{B}_{j+1}=S_{j+1}=\begin{bmatrix}   D_{j+1} &  \\
                                              E_{j+1}& T_{j+1,3}+C_{j+1,3} \\
                                            \end{bmatrix}
                                          \begin{matrix}
                                 \mathcal{H}_{j+1}^1\oplus\mathcal{H}_{j+1}^2 \\
                                 \mathcal{H}_{j+1}^3
                               \end{matrix}.\end{eqnarray*}

Let $$X=\begin{bmatrix}X_{11}&X_{12}\\
X_{21}&X_{22}\end{bmatrix}\begin{matrix}
                                 \mathcal{H}_{j+1}^1\oplus\mathcal{H}_{j+1}^2 \\
                                 \oplus\\
                                 \mathcal{H}_{j+1}^3
                               \end{matrix}\mapsto\begin{matrix}
                                 \mathcal{H}_j^1\oplus\mathcal{H}_j^2 \\
                                 \oplus\\
                                 \mathcal{H}_j^3
                               \end{matrix}$$
belongs to Ker$\mathcal{T}_{\overline{B}_{j},\overline{B}_{j+1}}$. Then \begin{eqnarray}
&&D_jX_{11}=X_{11}D_{j+1}+X_{12}E_{j+1},\\
&&D_jX_{12}=X_{12}(T_{j+1,3}+C_{j+1,3}),\\
&&E_jX_{11}+(T_{j3}+C_{j3})X_{21}=X_{21}D_{j+1}+X_{22}E_{j+1},\\
&&E_jX_{12}+(T_{j3}+C_{j3})X_{22}=X_{22}(T_{j+1,3}+C_{j+1,3}).\end{eqnarray}

Since $\sigma(D_j)\cap\sigma(T_{j+1,3}+C_{j+1,3})=\emptyset$, from lemma \ref{HER1} and $(2.9)$, $X_{12}=0$. Since $\sigma(T_{j3}+C_{j3})\cap\sigma(T_{j+1,3}+C_{j+1,3})=\emptyset$, from lemma \ref{HER1}, $X_{12}=0$ and $(2.11)$, $X_{22}=0$.

If $X_{11}=0$, then since $\sigma(T_{j3}+C_{j3})\cap\sigma(D_{j+1})=\emptyset$, $X_{21}=0$, and hence $X=0$.

Let $$X_{11}=\begin{bmatrix}                 \ddots&\vdots&\vdots&\vdots&\vdots&

{\mathinner{\mkern2mu\raise1pt\hbox{.}\mkern2mu
\raise4pt\hbox{.}\mkern2mu\raise7pt\hbox{.}\mkern1mu}}
 \\
                 \cdots&a_{-1,-1}&a_{-1,0}&a_{-1,1}&a_{-1,2}&\cdots\\
                \cdots&a_{0,-1}&a_{0,0}&a_{0,1}&a_{0,2}&\cdots\\
                \cdots&a_{1,-1}&a_{1,0} &a_{1,1}&a_{1,2}&\cdots\\
                \cdots&a_{2,-1}&a_{2,0}&a_{2,1}&a_{2,2}&\cdots\\

{\mathinner{\mkern2mu\raise1pt\hbox{.}\mkern2mu
\raise4pt\hbox{.}\mkern2mu\raise7pt\hbox{.}\mkern1mu}}

&\vdots&\vdots&\vdots&\vdots&\ddots\\
                                          \end{bmatrix}
                   \begin{matrix}
                      \vdots \\
                      e_{-1}^{j+1} \\
                      e_0^{j+1} \\
                      e_1^{j+1} \\
                      e_2^{j+1} \\
                      \vdots
                     \end{matrix}\mapsto\begin{matrix}
                      \vdots \\
                      e_{-1}^j \\
                      e_0^j \\
                      e_1^j \\
                      e_2^j \\
                      \vdots
                     \end{matrix},$$ then from $(2.8)$, we have \begin{eqnarray}&&
                     \gamma_m^ja_{m,n}=\gamma_n^{j+1}a_{m+1,n+1},~\forall m,n\in\mathbb{Z}.
                     \end{eqnarray}

Hence \begin{eqnarray}&&a_{m+1,m+k+1}=\dfrac{\gamma_0^j\gamma_1^j\cdots\gamma_m^j}{\gamma_k^{j+1}\gamma_{k+1}^{j+1}\cdots\gamma_{m+k}^{j+1}}a_{0,k},~\forall m\geq0,~k\in\mathbb{Z},\\
&&a_{m,m+k}=\dfrac{\gamma_{m+k}^{j+1}\gamma_{m+k+1}^{j+1}\cdots\gamma_{k-1}^{j+1}}{\gamma_m^j\gamma_{m+1}^j\cdots\gamma_{-1}^j}a_{0,k},~\forall m<0,~k\in\mathbb{Z}.\end{eqnarray}

Since $$\dfrac{\gamma_n^j}{\gamma_m^{j+1}}\geq\dfrac{\lambda_{j+1}}{\eta_{j+1}^2}>1,~\forall m, n\geq1,$$ $$\lim\limits_{m\rightarrow +\infty}\dfrac{\gamma_0^j\gamma_1^j\cdots\gamma_m^j}{\gamma_k^{j+1}\gamma_{k+1}^{j+1}\cdots\gamma_{m+k}^{j+1}}=\infty.$$
Hence from $(2.13)$ and $(2.14)$,  $a_{m,n}=0,~\forall m,n\in\mathbb{Z}$.
\end{proof}

\begin{remark}
If $T$ is a positive operator which satisfies $(0,\epsilon)\cap\sigma_e(T)=\emptyset$ for some $\epsilon>0$, Ker$T=\{0\}$ and $\sigma_e(T)$ is not connected, then there exist an unitary operator $U$ and a compact operator $K$ with norm less than a given positive number such that $(U+K)T$ is a strongly irreducible Cowen-douglas operator. If $T$ is a compact operator which is injective and has dense range, then for any unitary operator $U$ and any compact operator $K$, $(U+K)T$ can not be a Cowen-Douglas operator. It is an easy corollary of Case 1,2,3 of the above proof.
\end{remark}

\section{Application in Basis Theory}

\subsection{}In this subsection, we discuss existence of operators satisfying
both certain basis theory property and strongly irreducibility.

Recall that a sequence $\psi=\{f_{n}\}_{n=1}^{\infty}$ is called a \textsl{Schauder basis} of the Hilbert space $\mathcal{H}$
if and only if for every vector $x \in \mathcal{H}$ there exists a unique sequence $\{\alpha_{n}\}_{n=1}^{\infty}$ of complex numbers
such that the partial sum sequence $x_{k}=\sum_{n=1}^{k} \alpha_{n}f_{n}$ converges to $x$ in norm.

\begin{definition}
An $\omega \times \omega$ matrix $F$ is said to be a \textsl{Schauder matrix} if and only if the sequence
$\psi=\{f_{n}\}_{n=1}^{\infty}$ of its
column vectors in $\mathcal{H}$ comprises a Schauder basis.
\end{definition}

Note that from the definition, each column vector $f_{k}$ is automatically a $l^{2}-$sequence since it represents
a vector in $\mathcal{H}$.
Given an ONB $\varphi=\{e_{n}\}_{n=1}^{\infty}$ and a basis sequence
$\psi=\{f_{n}=\sum_{k=1}^{\infty} f_{kn}e_{k}\}_{n=1}^{\infty}$ of $\mathcal{H}$, then the matrix
$F_{\psi}=(f_{kn})$ is a Schauder matrix by above definition. It shall be called the Schauder matrix corresponding
to the basis $\psi$.

\begin{definition}
A matrix $F$ is called a \textsl{unconditional, conditional, Riesz, normalized} or \textsl{quasinormal} matrix respectively if and only if
the sequence of its column vectors comprise a unconditional, conditional, Riesz, normalized or quasinormal basis. Two Schauder
matrices $F_{\psi}, F_{\varphi}$ are called \textsl{equivalent} if and only if the corresponding bases $\psi$ and $\varphi$ are equivalent.
\end{definition}

\begin{theorem}\label{BPS}
Assume that $F$ is a Schauder matrix and $G^{*}$ is its inverse matrix. We have \\
1. For each invertible matrix $X$, $XF$ is also a Schauder matrix. Moreover, $XF$ is unconditional(conditional)
if and only if $F$ is unconditional(conditional);\\
2. For each diagonal matrix $D=diag(\alpha_{1}, \alpha_{2}, \cdots)$ in which each diagonal element $\alpha_{k}$ is nonzero,
$FD$ is also a Schauder matrix. Moreover, $FD$ is unconditional(conditional)
if and only if $F$ is unconditional(conditional); \\
3. For a unconditional matrix $F$, $FU$ is also a unconditional matrix for $U \in \pi_{\infty}$; \\
4. Two Schauder matrix $F$ and $F^{'}$ are equivalent if and only if there is a invertible matrix $X$ such
that $XF=F^{'}$.
\end{theorem}

For a Schauder matrix $M_{\psi}$ and a matrix $X$ of some invertible operator $T$, $M_{\psi^{'}}=XM_{\psi}$ is the Schauder matrix
of the Schauder basis $\psi^{'}=\{Tf_{n}\}_{n=1}^{\infty}$. Given an ONB $\varphi$, the orbit set
$$
O_{gl}(M_{\psi})=\{XM_{\psi}; \mbox{matrix $X$ represents an invertible operator $T \in Gl(\mathcal{H})$}\}
$$
consists of all equivalent bases of $\psi$. Our main theorem \ref{Theorem: XT can be strongly irreducible} and above theorem \ref{BPS}
tell us that we always can pick a strongly irreducible
operator as the representation element in the orbit $O_{gl}(M_{\psi})$.

\begin{theorem}\label{Theorem: XM can represents a strongly irreducible operator I}
Suppose that $\psi=\{f_{n}\}_{n=1}^{\infty}$ is a basis sequence and its corresponding
Schauder matrix $M_{\psi}$ under some given ONB represents a bounded operator $T_{\psi}$.
Then there is a matrix $X$ of some invertible operator such that $XM_{\psi}$ represents an strongly irreducible bounded compact operator in $L(\mathcal{H})$.
\end{theorem}

\begin{corollary}
For a Schauder matrix $M_{\psi}$ representing a bounded operator, there always be matrices of strongly irreducible
operators in its orbit $O_{gl}(M_{\psi})$.
\end{corollary}

By the pole decomposition theorem of operators, there is also a unitary matrix $U$ such that $UM_{\psi}$ represents a self-adjoint operator.
Hence the orbit $O_{gl}(M_{\psi})$ has both the operators having a large number of strongly reducible subspaces and
the operators having no nontrivial strongly reducible subspace.

\subsection{The blowing up matrix}
To observe the relations between basis theory and operator theory, is it enough to just consider bounded operators?
Let $\psi=\{f_{n}\}_{n=1}^{\infty}$ be a Schauder basis and $M_{\psi}$ is the corresponding Schauder matrix under some ONB
$\varphi=\{e_{n}\}_{n=1}^{\infty}$. In general, $M_{\psi}$ does not represents a bounded operator even for a quasinormal
basis $\psi$. Following example show this phenomenon.

\begin{example}\label{Example: QNBnotBounded}(see \cite{Singer}, Example14.5, p429.)
Let $\{\alpha_{n}\}_{n=1}^{\infty}$ be a sequence of positive numbers such that $\sum_{n=1}^{\infty}n\alpha_{n}^{2}<\infty$
and $\sum_{n=1}^{\infty} \alpha_{n}=\infty$. Then the sequence $\{f_{n}\}_{n=1}^{\infty}$ defined as
$$
f_{2n-1}=e_{2n-1}+\sum_{i=n}^{\infty} \alpha_{i-n+1}e_{2i}, ~~f_{2n}=e_{2n}, ~~n=1, 2, \cdots
$$
is a conditional basis of $\mathcal{H}$. Denote by $F$ the corresponding Schauder matrix and $T$ the operator it represented.
Now we shall show that $T$ is a unbounded operator indeed. To do this, we rewrite the matrix of $T$ under new ONB. Here we use the fact that
$T$ is bounded if and only if $U^{*}TU$ is bounded for each unitary operator $U$. Denote by $\mathcal{H}_{1}=span \{e_{2n-1}; n=1, 2, \cdots\}$
and $\mathcal{H}_{2}=span \{e_{2n}; n=1, 2, \cdots\}$. Now we can rewrite $F$ into the form
$$
\begin{array}{rll}
T=&
\left(
\begin{array}{cc}
I & 0 \\
T_{1} & I
\end{array}
\right) & \begin{array}{c}
           \mathcal{H}_{1} \\
           \mathcal{H}_{2}
          \end{array}
\end{array}
$$
in which the operator $T_{1}$ has a matrix as
$$
\left(
 \begin{array}{cccc}
  \alpha_{1}&0&0&0 \\
  \alpha_{2}&\alpha_{1}&0&0\\
  \alpha_{3}&\alpha_{2}&\alpha_{1}&0\\
  \vdots&\ddots&\ddots&\ddots
 \end{array}
\right).
$$
Denote by $S^{(2)}$ the shift operator defined as $S^{(2)}e_{2n}=e_{2n+2}$. It is trivial to check
that we have $S^{(2)}T_{1}=T_{1}S^{(2)}$. Hence we have
$T_{2}$ is in the commutant $\mathcal{A}^{'}(S^{(2)})$ if $T_{2}$ is a bounded operator. But it is impossible since the holomorphic
function defined by series $\sum_{n=1}^{\infty} \alpha_{n}z^{n}$ is not in the class $H^{\infty}$
by the fact $\sum_{n=1}^{\infty} \alpha_{n}=\infty$(cf, \cite{ALS}, theorem 3, p62).
\end{example}

However, we always can relate a basis to a bounded operator as follows.
Assume that $M=(f_{kn})$ is an $\omega \times \omega$ matrix such that each column
vector $f_{n}=\{f_{kn}\}_{k=1}^{\infty}$ is an $l^{2}-sequence$.
Then under a fixed ONB sequence $\{e_{n}\}_{n=1}^{\infty}$, $\mathbf{f}_{n}=\sum_{n=1}^{\infty}f_{kn}e_{k}$ is a vector in $\mathcal{H}$.
Then for a nonzero complex number sequence $\alpha=\{\alpha_{k}\}_{k=1}^{\infty}$, it is clear that
$M_{\alpha}=(\alpha_{1}f_{1}, \alpha_{2}f_{2}, \alpha_{3}f_{3}, \cdots)$ is also a matrix with $l^{2}-$sequences as its column vectors.
Moreover, if $M$ represents a bounded operator and $\alpha$ be a bounded sequence, then $M_{\alpha}$
also be a matrix of a bounded operator
since $M^{'}=MD$ in which $D$ is the bounded diagonal operator with $\alpha_{k}$ as its diagonal elements.
Inspired by the properties of bases(cf, proposition 4.1.5 and 4.2.12 in the book \cite{B}), we have the following definition.

\begin{definition}\label{Def: Blowing up Matrix}
$M_{\alpha}$ is called \textsl{the blowing up matrix of $M$ with sequence $\alpha$}.
\end{definition}

\begin{theorem}\label{Theorem: Each Basis has a Related Blowing Up Schauder Matrix}
Suppose that $\{f_{n}\}_{n=1}^{\infty}$ is a basis sequence and $\{\alpha_{n}\}_{n=1}^{\infty}$ is a sequence of complex
numbers such that the sum $\sum_{n=1}^{\infty} ||\alpha_{n}f_{n}||$ be finite. Moreover, let $\psi=\{\alpha_{n}f_{n}\}_{n=1}^{\infty}$.
Then the matrix $M_{\psi}$ represents a bounded compact operator under any ONB.
\end{theorem}
\begin{proof}
Fixed an ONB $\varphi=\{e_{n}\}_{n=1}^{\infty}$. Then each vector $f_{n}$ has a unique $l^{2}-$coodinate
$f_{n}=\{f_{kn}e_{k}\}_{k=1}^{\infty}$. Let $g_{n}=\{g_{kn}\}_{k=1}^{\infty}=\alpha_{n}f_{n}=\{\alpha f_{kn}e_{k}\}_{k=1}^{\infty}$.
Then the matrix $M_{\psi}=(g_{kn})$ under the ONB $\varphi$. For any vector $x=\sum_{k=1}^{\infty} \xi_{k}e_{k}$,
the series
$
\sum_{k=1}^{\infty} \xi_{k}g_{k}
$
converges since we have
$$
||\sum_{k=m}^{\infty} \xi_{k}g_{k}|| \le \sum_{k=m}^{\infty} ||\xi_{k}g_{k}||
\le \sup_{k}\{|\xi_{k}|\}(\sum_{k=m}^{\infty}||g_{k}||) \rightarrow 0
$$
as $m \rightarrow 0$. Hence $M_{\psi}$ represents an operator $T_{\psi}$ well-defined everywhere on $\mathcal{H}$. Now by the
closed graph theorem, we just need to show that $M_{\psi}$ also represents a closed operator to finish the proof.
Suppose that $x_{n}=\{\xi_{k}^{(n)}\}_{k=1}^{\infty} \rightarrow x_{0}=\{\xi_{k}^{(0)}\}_{k=1}^{\infty}$ in the norm.
Now for any $\epsilon >0$, there is some integer $n_{0}$ such that $|\xi^{(n)}_{k}-\xi^{(0)}_{k}|<1$ holds for any $n>n_{0}$.
Let $k_{1}$ be the integer satisfying $\sum_{k=k_{1}+1}^{\infty} ||g_{k}||<\frac{\epsilon}{2}$. Then there is an integer
$n_{1}$ such that we have $\sum_{k=1}^{k_{1}} |\xi_{k}^{(n)}-\xi_{k}^{(0)}|||g_{k}||<\frac{\epsilon}{2}$
for all $n>n_{1}$. Let $N=\max\{n_{0}, n_{1}\}$.
Hence for $n>N$ we have
$$
\begin{array}{rl}
||T_{\psi}(x_{n}-x_{0})||& =||\sum_{k=1}^{\infty}(\xi_{k}^{(n)}-\xi_{k}^{(0)})g_{k}|| \\
                         & \le ||\sum_{k=1}^{k_{1}}(\xi_{k}^{(n)}-\xi_{k}^{(0)})g_{k}||+||\sum_{k=k_{1}+1}^{\infty}(\xi_{k}^{(n)}-\xi_{k}^{(0)})g_{k}|| \\
                         & \le \sum_{k=1}^{k_{1}} |\xi_{k}^{(n)}-\xi_{k}^{(0)}|||g_{k}||
                         + \sum_{k=k_{1}+1}^{\infty} ||g_{k}||<\frac{\epsilon}{2} \\
                         & \le \frac{\epsilon}{2}+\frac{\epsilon}{2}.
\end{array}
$$

Now it is left to show that $M_{\psi}$ is also a compact operator. Let $K_{n}=\sum_{l=1}^{n} g_{l} \otimes e_{l}$ in which
the operator $g_{l} \otimes e_{l}$ is defined as $(g_{l} \otimes e_{l})x=(x, e_{l})g_{l}$. Then we have
$$
\begin{array}{rl}
||(T_{\psi}-K_{n})||&=\sup_{||x||=1}||(T_{\psi}-K_{n})x|| \\
              &=\sup_{||x||=1}||(\sum_{l=n+1}^{\infty} g_{l} \otimes e_{l})x|| \\
              &\le \sup_{||x||=1}\sum_{l=n+1}^{\infty} ||g_{l} \otimes e_{l})x|| \\
              &\le \sum_{l=n+1}^{\infty} ||g_{l}|| \rightarrow 0 \\
\end{array}
$$
as $n \rightarrow \infty$. Hence $T_{\psi}$ is a norm limit of a sequence of compact operators, that is, $T_{\psi}$
is a compact operator.
\end{proof}

Above theorem enlarge the class of bases that can be studied by bounded operators.

\begin{theorem}\label{Theorem: XM can represents a strongly irreducible operator II}
Suppose that $\{f_{n}\}_{n=1}^{\infty}$ is a basis sequence and $\{\alpha_{n}\}_{n=1}^{\infty}$ is a sequence of complex
numbers such that the sum $\sum_{n=1}^{\infty} ||\alpha_{n}f_{n}||$ be finite. Moreover, let $\psi=\{\alpha_{n}f_{n}\}_{n=1}^{\infty}$.
Then there is a matrix $X$ of some invertible operator such that $XM_{\psi}$ represents an strongly irreducible bounded compact operator in $L(\mathcal{H})$.
\end{theorem}

\subsection{}
Now we turn to study the existence of strongly irreducible Schauder operator.
Recall that a Schauder operator $T$ is an operator mapping
some ONB into a Schauder basis.
In his paper \cite{Ole}, Olevskii call an operator to be \textsl{generating} if and only if it maps some ONB
into a quasinormal conditional basis. Hence our definition of Schauder operator is a generalization of Olevskii's
one. A Schauder operator is said to be conditional if it maps some ONB sequence into a conditional basis sequence.
Our theorem \ref{Theorem: XT can be strongly irreducible} ensure the existence of operators satisfying special properties.
For convenience and self-sufficiency, we list some results on an operator theory description
of Schauder basis appearing in paper \cite{Cao-Tian-Hou} without proof.

\begin{theorem}\label{BPSO}
Following conditions are equivalent:\\
1. $T$ is a Schauder operator;\\
2. $T$ maps some ONB $\{e_{n}\}_{n=1}^{\infty}$ into a basis;\\
3. $T$ has a polar decomposition $T=UA$ in which $A$ is a Schauder operator;\\
4. Assume that $T$ has a matrix representation $F$ under a fixed ONB $\{e_{n}\}_{n=1}^{\infty}$. There
is some unitary matrix $U$ such that $FU$ is a Schauder matrix.
\end{theorem}

\begin{corollary}\label{Corollary: A conditional diagonal operator}(\cite{Cao-Tian-Hou})
Compact operator $K=diag\{1,\frac{1}{2},\frac{1}{3},\cdots\}$ is a conditional operator.
\end{corollary}

\begin{corollary}
There is an operator $T \in L(\mathcal{H})$ satisfying following properties:\\
1. $T$ is a strongly irreducible compact operator;\\
2. There exists some ONB such that $T$ has a matrix which is a conditional Schauder matrix.
\end{corollary}
\begin{proof}
Fixed an ONB, suppose that $K=diag\{1,\frac{1}{2},\frac{1}{3},\cdots\}$ be the diagonal operator appearing in
corollary \ref{Corollary: A conditional diagonal operator}. Then there is some unitary matrix $U$ such that
$KU$ is a conditional Schauder matrix. Then apply theorem \ref{Theorem: XT can be strongly irreducible} and theorem \ref{BPSO}
to the matrix $KU$.
\end{proof}

\begin{corollary}
There is an operator $T \in L(\mathcal{H})$ satisfying following properties:\\
1. $T$ is a strongly irreducible compact operator;\\
2. There exists some ONB such that $T$ has a matrix which is a unconditional Schauder matrix.
\end{corollary}
\begin{proof}
Fixed an ONB, suppose that $K=diag\{1,\frac{1}{2},\frac{1}{3},\cdots\}$ be the diagonal operator appearing in
corollary \ref{Corollary: A conditional diagonal operator}. Then apply theorem \ref{Theorem: XT can be strongly irreducible} and theorem \ref{BPSO} to the matrix $K$.
\end{proof}

By the following main theorem 1 of \cite{Ole}, we can get the same results about the non-compact case.

\begin{theorem}\label{Theorem: Olevskii}(\cite{Ole}, p479)
A bounded operator $T$ is generating if and only if the following conditions hold:

a) the operators $T$ and $T^*$ do not admit the eigenvalue $\lambda=0$;

b) there exists a number $q,0<q<1$ such that to each segment $[q^{n+1},q^n]$ in the spectral decomposition of the positive operator $(T^*T)^{\frac{1}{2}}$, there corresponds an infinite dimensional invariant subspace.
\end{theorem}

\begin{corollary}
There is an operator $T \in L(\mathcal{H})$ satisfying following properties:\\
1. $T$ is a strongly irreducible non-compact operator;\\
2. There exists some ONB such that $T$ has a matrix which is a conditional Schauder matrix.
\end{corollary}
\begin{proof}
Let $T$ be a self-adjopint operator satisfying the conditional b). Then there is some unitary matrix $U$ such that
$TU$ is a conditional Schauder matrix. It is clear that $TU$ can not be compact. Then apply theorem \ref{Theorem: XT can be strongly irreducible} and theorem \ref{BPSO} to the operator $T$.
\end{proof}

\begin{corollary}
There is an operator $T \in L(\mathcal{H})$ satisfying following properties:\\
1. $T$ is a strongly irreducible non-compact operator;\\
2. There exists some ONB such that $T$ has a matrix which is a unconditional Schauder matrix.
\end{corollary}
\begin{proof}
In fact by proposition 2.20 in paper \cite{Cao-Tian-Hou} and theorem \ref{Theorem: XT can be strongly irreducible}, all
invertible and strongly irreducible operators satisfy the requirements of corollary. For example, given an irreducible operator
$T$ and a complex number $\lambda \notin \sigma(T)$, then $\lambda-T$ is such one operator.
\end{proof}

The first requirement of above lemmas are in the operator theory category while the others are in basis theory
category. The first one ask that the matrix be good in the sense of Jordan block; the second one ask that
the matrix be nice from the basis viewpoint, that is, the column vectors comprise a basis.

\subsection{Small disturbance on the basis const}

Now we consider the ``small'' condition appearing in the theorem \ref{Theorem: XT can be strongly irreducible}, it ensure us
to get a new basis with a small disturbance on both of the basis const and the unconditional const.

Recall that a sequence $\psi=\{f_{n}\}_{n=1}^{\infty}$ is called a Schauder basis of the Hilbert space $\mathcal{H}$
if and only if for every vector $x \in \mathcal{H}$ there exists a unique sequence $\{\alpha_{n}\}_{n=1}^{\infty}$ of complex numbers
such that the partial sum sequence $x_{k}=\sum_{n=1}^{k} \alpha_{n}f_{n}$ converges to $x$ in norm.

Denote by $P_{k}$ the the diagonal matrix with the first $k-$th entries on diagonal line equal to 1 and $0$ for others.
Then as an operator $P_{k}$ represents the orthogonal projection
from $\mathcal{H}$ to the subspace $\mathcal{H}^{(k)}=span \{e_{1}, e_{2}, \cdots, e_{k}\}$.

\begin{lemma}\label{Matrix Form}
Assume that $\{e_{k}\}_{k=1}^{\infty}$ is a fixed ONB of $\mathcal{H}$.
Suppose that an $\omega \times \omega$ matrix $F=(f_{ij})$
satisfies the following properties:

1. Each column of the matrix $F$ is a $l^{2}-$sequence;

2. $F$ has a unique left inverse matrix $G^{*}=(g_{kl})$ such that each row of $G^{*}$ is also a $l^{2}-$sequence;

3. Operators $Q_{k}$ defined by the matrix $Q_{k}=FP_{k}G^{*}$ are well-defined projections on $\mathcal{H}$
and converges to the unit operator $I$ in the strong operator
topology.

Then the sequence $\{f_{k}\}_{k=1}^{\infty}, f_{k}=\sum_{j=1}^{\infty} f_{ij}e_{i}$ must be a Schauder basis.
\end{lemma}

The projection $FP_{n}G^{*}$ is just the $n-$th ``natural projection'' so called in \cite{B}(p354).
It is also the \text{$n-$th partial sum operator} so called in \cite{Singer}(definition 4.4, p25). Now we can translate
theorem 4.1.15 and corollary 4.1.17 in \cite{B} into the following
\begin{proposition}\label{BC1}
If $F$ is a Schauder matrix, then $M=\sup_{n} \{||FP_{n}G^{*}||\}$ is a finite const.
\end{proposition}

The const $M$ is called the \textsl{basis const} for the basis $\{f_{n}\}_{n=1}^{\infty}$.

Assume that $\psi=\{f_{n}\}_{n=1}^{\infty}$ is a basis.
For a subset $\Delta$ of $\mathbb{N}$, denote by $P_{\Delta}$ the diagonal matrix defined as
$P_{\Delta}(nn)=1$ for $n \in \Delta$ and $P_{\Delta}(nn)=0$ for $n \notin \Delta$.
The projection $Q_{\Delta}=F_{\psi}P_{\Delta}G_{\psi}^{*}$ defined in above lemmas
is called a \textsl{natural projection}(see, definition 4.2.24, \cite{B}, p378).
In fact for a vector $x=\sum_{n=1}^{\infty} x_{n}f_{n}$, it is trivial to check $Q_{\Delta} x=\sum_{n \in \Delta} x_{n}f_{n}$.
Then we have a same result for the \textsl{ unconditional basis const}(cf, definition4.2.28, \cite{B}, p379):
\begin{proposition}\label{UBC1}
If $F_{\psi}$ is a Schauder matrix, then the unconditional basis
const of the basis $\psi$ is $M_{ub}=\sup_{\Delta \subseteq \mathbb{N}} \{||F_{\psi}P_{\Delta}G_{\psi}^{*}||\}$.
\end{proposition}

In virtue of the proposition 4.2.29 and theorem 4.2.32 in the book \cite{B}, we have
\begin{proposition}\label{Proposition: Unconditional Matrix}
For a Schauder basis $\psi$, it is a unconditional basis
if and only if $\sup_{\Delta \subseteq \mathbb{N}} \{||F_{\psi}P_{\Delta}G_{\psi}^{*}||\}<\infty$.
\end{proposition}

Now we can show that we can change a Schauder matrix to an equivalent Schauder matrix which represents a strongly irreducible operator
with small disturbance on basis const.

\begin{theorem}\label{Theorem:Small disturbance on basis const}
Suppose that $\psi=\{f_{n}\}_{n=1}^{\infty}$ be a basis and its corresponding Schauder matrix $F_{\psi}$ under
a given ONB $\varphi=\{e_{n}\}_{n=1}^{\infty}$ represents
a bounded operator $T_{\psi}$ in $L(\mathcal{H})$.  Then for any positive number $\epsilon >0$, there is a matrix $X$ of some invertible
operator $T \in L(\mathcal{H})$ such that following properties hold: \\
1. $XF_{\psi}$ is a strongly irreducible operator; \\
2. The column vector sequence $\psi^{'}$ of $XF_{\psi}$ comprise a basis sequence equivalent to $\psi$. Moreover, if we denote by
$M^{'}$ the basis const of $\psi^{'}$, then we can ask that the condition $|M-M^{'}|<\epsilon$ holds;
If $\psi$ is also a unconditional basis, we can also ask that the additional condition
$|M_{ub}^{'}-M_{ub}|<\epsilon$ holds(Here $M_{ub}^{'}$ denote the unconditional basis const of the basis $\psi^{'}$).
\end{theorem}
\begin{proof}
By our main theorem \ref{Theorem: XT can be strongly irreducible}, we only need to show that the second property holds
when $X=U+K$ been chosen carefully. Clearly for a given $\delta>0$, we can choose a $X=U+K$ such that both
$1-\frac{\delta}{2}<||U+K||<1+\frac{\delta}{2}$ and $1-\frac{\delta}{2}<||(U+K)^{-1}||<1+\frac{\delta}{2}$ hold.
Then we have
$$
\begin{array}{c}
||XF_{\psi}P_{k}G^{*}_{\psi}X^{-1}|| \le ||X||\cdot ||F_{\psi}P_{k}G^{*}_{\psi}||\cdot ||X^{-1}||, \\
||F_{\psi}P_{k}G^{*}_{\psi}|| \le ||X^{-1}||\cdot ||XF_{\psi}P_{k}G^{*}_{\psi}X^{-1}||\cdot ||X||.
\end{array}
$$
Hence following inequality holds for any $k \in \mathbb{N}$:
$$
(1+\delta)^{-2} \cdot ||F_{\psi}P_{k}G^{*}_{\psi}|| \le ||XF_{\psi}P_{k}G^{*}_{\psi}X^{-1}|| \le (1+\delta)^{2}\cdot ||F_{\psi}P_{k}G^{*}_{\psi}||.
$$
Now by proposition \ref{BC1}, we have
$$
M=\sup_{n} \{||FP_{n}G^{*}||\} \hbox{ and } M^{'}=\sup_{n} \{||XFP_{n}G^{*}X^{-1}||\}.
$$
Therefore by choosing $\delta < \frac{\epsilon}{4M}$ we can prove the first part of property 2; the second part of property 2
can be proved in just the same way by lemma \ref{Proposition: Unconditional Matrix}.
\end{proof}

{\bf Acknowledgements} A large part of this article was developed
during the seminar on operator theory held at Jilin University in China.

\end{document}